\documentclass[11pt]{article}
\usepackage{amsmath,fullpage,amssymb,amsthm,hyperref}
\usepackage{youngtab,tikz}

\usepackage{color}
\usepackage[numbers,sort&compress]{natbib}

\newtheorem{theorem}{Theorem}[section]
\newtheorem{prop}[theorem]{Proposition}
\newtheorem{lemma}[theorem]{Lemma}
\newtheorem{corollary}[theorem]{Corollary}

\def\LTR{left-to-right }
\newcommand\beq{\begin{equation}}
\newcommand\eeq{\end{equation}}
\def\P{\mathcal P}

\def\S{\mathcal S}
\def\I{\mathcal I}
\def\D{\mathcal D}
\def\G{\mathcal G}
\DeclareMathOperator\asc{asc}
\DeclareMathOperator\Asc{Asc}
\DeclareMathOperator\des{des}
\DeclareMathOperator\Des{Des}
\DeclareMathOperator\maj{maj}
\DeclareMathOperator\comaj{comaj}
\DeclareMathOperator\hd{HD}

\DeclareMathOperator\Peak{Peak}
\def\fn2{\lfloor \frac{n}{2} \rfloor}
\def\cn2{\lceil \frac{n}{2} \rceil}

\def\AR{\rho}
\def\hookpeak{\psi}
\def\x{\mathbf{x}}
\def\urcorner{\left(\cn2,\fn2\right)}
\DeclareMathOperator\SYT{SYT}

\begin{document}

\author{Marilena Barnabei \\
Dipartimento di Matematica \\
Bologna, 40126, ITALY \\
\texttt{marilena.barnabei@unibo.it}\and
Flavio Bonetti \\
Dipartimento di Matematica \\
Bologna, 40126, ITALY \\
\texttt{flavio.bonetti@unibo.it}\and
Sergi Elizalde\thanks{Corresponding author. Phone: +1-603-646-8191.} \\
Department of Mathematics\\
Dartmouth College \\
Hanover, NH 03755, USA \\
\texttt{sergi.elizalde@dartmouth.edu}
\and
Matteo Silimbani \\
Dipartimento di Matematica \\
Bologna, 40126, ITALY \\
\texttt{matteo.silimbani4@unibo.it}}

\date{}

\title{Descent sets on $321$-avoiding involutions and hook decompositions of partitions}

\maketitle

\begin{abstract}
We show that the distribution of the major index over the set of involutions in $\S_n$ that avoid the pattern $321$ is given
by the $q$-analogue of the $n$-th central binomial coefficient. The proof consists of a composition of three non-trivial bijections, one being the Robinson-Schensted correspondence, ultimately mapping those
involutions
with major index $m$ into partitions of $m$ whose Young diagram fits inside a $\fn2 \times \cn2$ box. We also obtain a refinement that keeps track of the descent set, and we deduce an analogous result for the comajor index of $123$-avoiding involutions.
\end{abstract}



\section{Introduction}

The study of statistics on pattern-avoiding permutations is an active area of research. In one of the first papers in this area, Robertson, Saracino and Zeilberger~\cite{RSZ} considered the number of fixed points and excedances
in permutations avoiding patterns of length 3, which sparked further work on these statistics by several authors~\cite{BloSar1,Eli,EliPak,EliDeu,Eli04}. More recently, other statistics such as the number of descents~\cite{BBS3,BBS4}, the major index and the number of inversions~\cite{SagSav,DDJSS,CEKS} have been studied on restricted permutations. Many of these papers show that certain statistics have the same distribution on permutations avoiding
different patterns, and in some cases they give this distribution.

There has also been a significant amount of work on pattern-avoiding involutions. Recall that an involution is a permutation that equals its inverse.
In one of the most cited papers on pattern avoidance, Simion and Schmidt~\cite{SS} count involutions avoiding each pattern of length $3$.
Other more recent papers consider various statistics on pattern-avoiding involutions~\cite{BMSte,DRS,BloSar2}. The present paper focuses on the descent number and major index statistic on $321$-avoiding involutions, and more generally on the distribution of the descent set.
Similar problems on unrestricted involutions have been well studied~\cite{Strehl,Guoz,Chow,BBS1}.

One novelty of our work is that we find a surprising connection between pattern-avoiding involutions and integer partitions.
Our main result is that descent sets on $321$-avoiding involutions have the same distribution as certain hook lengths on partitions whose Young diagram fits inside a box.
In particular, the major index statistic translates to the area of the Young diagram. We obtain a bijective proof by composing three non-trivial statistic-preserving bijections, first going from permutations to lattice paths and then to partitions. One peculiarity of our main result is that when the length of the permutation is large enough in comparison with the largest descent,
one can give a much simpler proof (discussed in Section~\ref{sec:largen}), which does not seem to extend to all cases.

In Section~\ref{sec:paths} we introduce some background on lattice paths, as well as one of the three pieces of the main bijection.
In Section~\ref{sec:321involutions} we state and proof the main results about descents and major index on $321$-avoiding involutions, presenting the two remaining pieces of the bijection, one of which involves the Robinson-Schensted correspondence and has been used in~\cite{EliPak}, and the other one which is new to the best of our knowledge.

In Section~\ref{sec:consequences} we discuss some consequences and extensions. We show that, using ideas from~\cite{Strehl}, our results extend to the ascent distribution on $123$-avoiding involutions. We also consider descents on involutions avoiding two patterns of length $3$. Finally, in Section~\ref{sec:321permutations} we turn to the larger set of all $321$-avoiding permutations, and we obtain formulas enumerating those with a given descent set.

Finally, Section~\ref{sec:Stanley} discusses an alternate proof of our result about the distribution of the major index on $321$-avoiding involutions. This proof uses symmetric functions, and it is not bijective, unlike the one provided in Section~\ref{sec:maj}.

\section{Lattice paths}\label{sec:paths}

An important tool in our study of pattern-avoiding involutions will be lattice paths. In this section we define the paths that we will use and we give some background.
Unless explicitly stated otherwise, all the paths in this paper are lattice paths with steps $N=(0,1)$ and $E=(1,0)$ starting at the origin $(0,0)$. The length of a path is its number of steps.

A {\em Dyck path} is a path ending on the line $y=x$ and not going below $y=x$. Denote by $\D_n$ the set of Dyck paths of length $2n$.
A {\em Dyck path prefix} (sometimes called {\em ballot path}) is a path not going below $y=x$. We denote by $\P_n$ the set of Dyck path prefixes of length $n$.
A {\em Grand Dyck path} of length $n$ is a path ending at $\urcorner$. We denote by $\G_n$ the set of Grand Dyck paths of length $n$. Note that this definition is more general than the standard one, which considers only Grand Dyck paths with an even number of steps.

A {\em peak} in a path is an occurrence of $NE$, which we sometimes identify with the vertex in the middle of such an occurrence.
If we label the vertices of path $P\in\P_n$ or $P\in\G_n$ from $0$ to $n$ starting at the origin, the {\em peak set} of $P$, denoted $\Peak(P)$, is the set of labels of the vertices that are peaks. For example, the peak set of both paths in Figure~\ref{fig:xi} is $\{2,6,9,14\}$.

Next we describe a bijection $\xi$ between $\P_n$ and $\G_n$, which belongs to mathematical folklore. A very similar construction was used by Greene and Kleitman~\cite{GreKle}
to give a symmetric chain decomposition of the boolean algebra, and also more recently by Elizalde and Rubey~\cite{EliRub} in the context of lattice paths.

Given $P\in\P_n$, match $N$s and $E$s that face each other, in the sense that the line segment (called a {\em tunnel} in~\cite{Eli}) from the midpoint of $N$ to the midpoint of $E$ has slope $1$ and stays below the path. Figure~\ref{fig:xi} shows an example. Thinking of the $N$s as opening parentheses and the $E$s as closing parentheses, the matched parentheses properly close each other.
Let $j$ be the number of unmatched steps, which are necessarily $N$ steps, since $P\in\P_n$. Note that $j$ and $n$ have the same parity. To obtain $\xi(P)$, change the first $\lceil\frac{j}{2}\rceil$ unmatched $N$ steps into $E$ steps.

It is clear that $\xi(P)\in\G_n$, since it has $\cn2$ $E$ steps and $\fn2$ $N$ steps. The inverse map is obtained again by matching $N$s and $E$s that face each other in the Grand Dyck path, and then changing all the unmatched $E$s (which necessarily come before the unmatched $N$s) into $N$s.

\begin{figure}[htb]
  \begin{center}
    \begin{tikzpicture}[scale=0.55]
      \draw (0,0) coordinate(d0)
      -- ++(0,1) coordinate(d1)
      -- ++(0,1) coordinate(d2)
      -- ++(1,0) coordinate(d3)
      -- ++(1,0) coordinate(d4)
      -- ++(0,1) coordinate(d5)
      -- ++(0,1) coordinate(d6)
      -- ++(1,0)  coordinate(d7)
      -- ++(0,1) coordinate(d8)
      -- ++(0,1) coordinate(d9)
      -- ++(1,0)  coordinate(d10)
      -- ++(1,0)  coordinate(d11)
      -- ++(0,1) coordinate(d12)
      -- ++(0,1) coordinate(d13)
      -- ++(0,1) coordinate(d14)
      -- ++(1,0)  coordinate(d15)
      -- ++(0,1) coordinate(d16);
      \foreach \x in {0,...,16} {
        \draw (d\x) circle (1.5pt);
      }
      \draw[dotted] (0,0)--(8,8);
      \draw[dashed] (0,0.5)--(1.5,2);
      \draw[dashed] (0,1.5)--(0.5,2);
      \draw[dashed] (2,3.5)--(2.5,4);
      \draw[dashed] (3,4.5)--(4.5,6);
      \draw[dashed] (3,5.5)--(3.5,6);
      \draw[dashed] (5,8.5)--(5.5,9);
      \draw[red,ultra thick] (d4)--(d5);
      \draw[red,ultra thick] (d11)--(d12);
      \draw[blue] (d12)--(d13);
      \draw[blue] (d15)--(d16);
      \draw(d2) node[left] {$2$};
      \draw(d6) node[left] {$6$};
      \draw(d9) node[left] {$9$};
      \draw(d14) node[left] {$14$};
      \draw (9,4) node [label=$\xi$] {$\mapsto$};
       \draw (11,0) coordinate(d0)
      -- ++(0,1) coordinate(d1)
      -- ++(0,1) coordinate(d2)
      -- ++(1,0) coordinate(d3)
      -- ++(1,0) coordinate(d4)
      -- ++(1,0) coordinate(d5) %
      -- ++(0,1) coordinate(d6)
      -- ++(1,0)  coordinate(d7)
      -- ++(0,1) coordinate(d8)
      -- ++(0,1) coordinate(d9)
      -- ++(1,0)  coordinate(d10)
      -- ++(1,0)  coordinate(d11)
      -- ++(1,0) coordinate(d12) %
      -- ++(0,1) coordinate(d13)
      -- ++(0,1) coordinate(d14)
      -- ++(1,0)  coordinate(d15)
      -- ++(0,1) coordinate(d16);
      \foreach \x in {0,...,16} {
        \draw (d\x) circle (1.5pt);
      }
      \draw[dotted] (11,0)--(19,8);
      \draw[dashed] (11,0.5)--(12.5,2);
      \draw[dashed] (11,1.5)--(11.5,2);
      \draw[dashed] (14,2.5)--(14.5,3);
      \draw[dashed] (15,3.5)--(16.5,5);
      \draw[dashed] (15,4.5)--(15.5,5);
      \draw[dashed] (18,6.5)--(18.5,7);
      \draw[red,ultra thick] (d4)--(d5);
      \draw[red,ultra thick] (d11)--(d12);
      \draw[blue] (d12)--(d13);
      \draw[blue] (d15)--(d16);
      \draw(d2) node[left] {$2$};
      \draw(d6) node[left] {$6$};
      \draw(d9) node[left] {$9$};
      \draw(d14) node[left] {$14$};
    \end{tikzpicture}
  \end{center}
  \caption{The bijection $\xi:\P_n\to\G_n$. The unmatched steps changed by $\xi$ are thicker and colored in red. The labels indicate the positions of the peaks, which are preserved by the bijection.}\label{fig:xi}
\end{figure}
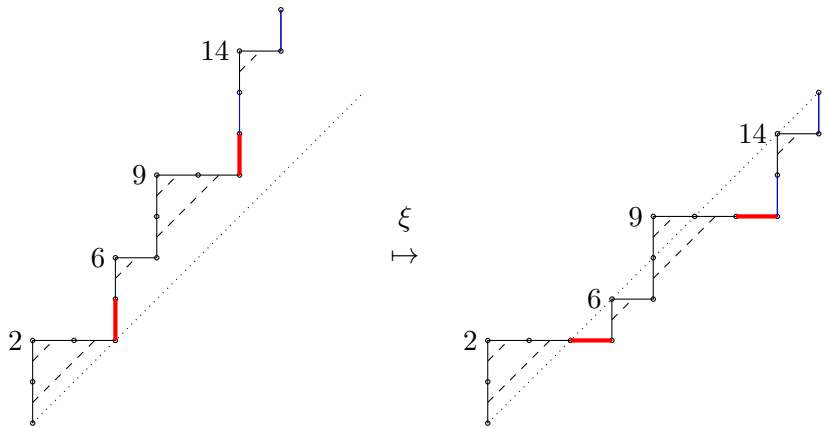

\begin{lemma}\label{lem:peaksPG}
For every $P\in\P_n$, we have
$$\Peak(P)=\Peak(\xi(P)).$$
\end{lemma}

\begin{proof}
For each peak $NE$ in a Dyck path prefix or Grand Dyck path, the steps $N$ and $E$ forming the peak are matched to each other, so peaks remain unchanged when applying $\xi$ or $\xi^{-1}$.
\end{proof}

\section{$321$-avoiding involutions}\label{sec:321involutions}

This section contains the main results of the paper, which concern statistics on $321$-avoiding involutions.

Let $\S_n$ (resp. $\I_n$) denote the set of permutations (resp. involutions) of $\{1,2,\dots,n\}$.
Recall that a permutation $\pi$ is an involution if $\pi=\pi^{-1}$.
A permutation $\pi(1)\dots\pi(n)$ is $321$-avoiding if there exist no $i<j<k$ such that $\pi(i)>\pi(j)>\pi(k)$.
Denote by $\S_n(321)$ (resp. $\I_n(321)$) the set of $321$-avoiding permutations (resp. involutions) in $\S_n$.

We say that a permutation $\pi$ has a \emph{descent} at position $i$, where $1\le i<n$, if $\pi(i)>\pi(i+1)$. Otherwise, we say that $\pi$ has an \emph{ascent} at that position. The set of descent positions of $\pi$ will be denoted by $\Des(\pi)$, while $\Asc(\pi)$ will denote the set of ascent positions. Moreover, we denote by $\des(\pi)$ and $\asc(\pi)$ the cardinalities of $\Des(\pi)$ and $\Asc(\pi)$, respectively. The sum of the entries in $\Des(\pi)$ is called the \emph{major index} of $\pi$:
$$\maj(\pi)=\sum_{i\in \Des(\pi)}i.$$
Similarly, the \emph{comajor index} of $\pi$ is the sum
$$\comaj(\pi)=\sum_{i\in \Asc(\pi)}i.$$

\subsection{Number of descents}\label{sec:des}

Our first goal is to give the distribution of the number of descents on $321$-avoiding involutions.
We start by describing a bijection $\AR$ between $\I_n(321)$ and $\P_n$ which, without the restriction to involutions, appears in~\cite[Section~3]{EliPak}, in~\cite{AdiRoi}, and in a similar form in~\cite[p. 64]{Knu}.

Given $\pi\in\I_n(321)$, we first apply
the Robinson-Schensted algorithm (see \cite[Section~3.1]{Sag}) to obtain a pair of
\emph{standard Young tableaux} of the same shape. By the symmetry of this algorithm (see~\cite{Schu}), the fact that $\pi=\pi^{-1}$
translates into the fact that these two tableaux are identical. Denote the resulting tableau by $Q$. Since $\pi$ avoids $321$, this tableau has at most two rows (by {\it Schensted's Theorem} \cite{Sche}, the number of rows equals the length of the longest decreasing subsequence of $\pi$). Thus, the Robinson-Schensted algorithm gives a bijection $\pi\mapsto Q$ between $\I_n(321)$ and the set of standard Young tableaux with $n$ boxes and at most 2 rows.

The tableau $Q$ can be interpreted as a Dyck path prefix, by letting the entries in the first row determine the positions of the $N$ steps, and the entries in the second row determine the positions of the $E$ steps.
Define $\AR(\pi)$ to be this Dyck path prefix. Figure~\ref{fig:AR} shows an example of the bijection $\AR$.

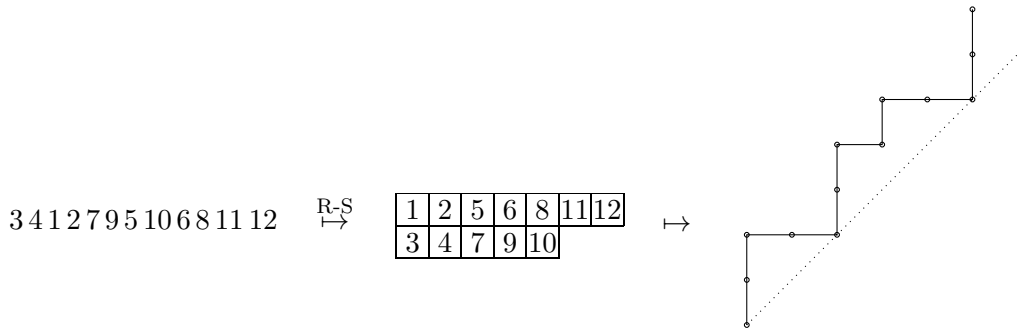
\begin{figure}[htb]
\newcommand\ten{10}
\newcommand\eleven{11}
\newcommand\twelve{12}
  \begin{center}
    \begin{tikzpicture}[scale=0.6]
      \draw (-3.5,2.5) node[left] {$3\,4\,1\,2\,7\,9\,5\,10\,6\,8\,11\,12\quad \stackrel{\text{R-S}}{\mapsto}$};
      \draw (-3,2.2) node[right] {$\young(12568\eleven\twelve,3479\ten)$};
      \draw (4,2.2) node[left] {$\mapsto$};
      \draw (5,0) coordinate(d0)
      -- ++(0,1) coordinate(d1)
      -- ++(0,1) coordinate(d2)
      -- ++(1,0) coordinate(d3)
      -- ++(1,0) coordinate(d4)
      -- ++(0,1) coordinate(d5)
      -- ++(0,1) coordinate(d6)
      -- ++(1,0)  coordinate(d7)
      -- ++(0,1) coordinate(d8)
      -- ++(1,0) coordinate(d9)
      -- ++(1,0)  coordinate(d10)
      -- ++(0,1)  coordinate(d11)
      -- ++(0,1) coordinate(d12);
      \foreach \x in {0,...,12} {
        \draw (d\x) circle (1.5pt);
      }
      \draw[dotted] (5,0)--(11,6);
    \end{tikzpicture}
  \end{center}
  \caption{The bijection $\AR:\I_n(321)\to\P_n$. }\label{fig:AR}
\end{figure}

We now show that the distribution of the descent set on $321$-avoiding involutions is the same as the distribution of the peak set on Dyck path prefixes.

\begin{lemma}\label{lem:DesPeak}
For every $\pi\in\I_n(321)$, we have $$\Des(\pi)=\Peak(\AR(\pi)).$$
\end{lemma}

\begin{proof}
Let $Q$ be the tableau obtained by applying the Robinson-Schensted algorithm to $\pi$. A property of this algorithm (see \cite[Remarque 2]{Schu} and \cite[Lemma 7.23.1]{EC2}) is that $\Des(\pi)$ equals the descent set of $Q$, that is,
the set of indices $i$ such that $i$ appears in the top row of $Q$ and $i+1$ appears in the bottom row.
This is equivalent to the $i$-th step of $\AR(\pi)$ being a $N$ step immediately followed by an $E$ step, namely, a peak.
\end{proof}

\begin{theorem}\label{thm:des}
For every $0\le k<n$,
$$|\{\pi\in\I_n(321):\des(\pi)=k\}|=\binom{\cn2}{k}\binom{\fn2}{k}.$$
\end{theorem}

\begin{proof}
By Lemmas~\ref{lem:DesPeak} and~\ref{lem:peaksPG}, the composition $\xi\circ\AR$ is a bijection between $\I_n(321)$ and $\G_n$ with the property that if $\pi\in\I_n(321)$ and $P=\xi(\AR(\pi))\in\G_n$, then
$\Des(\pi)=\Peak(P)$, and in particular $\des(\pi)=|\Peak(P)|$. Thus, it is enough to find the number of paths in $\G_n$ with $k$ peaks.
This number equals $$\binom{\cn2}{k}\binom{\fn2}{k},$$
since such a path is uniquely determined by the coordinates of its peaks $(x_1,y_1),\dots,(x_k,y_k)$ with $x_1<\dots<x_k$ and $y_1<\dots<y_k$, where
the $x$-coordinates are an arbitrary subset of $\{0,1,\dots,\cn2-1\}$ and the $y$-coordinates are an arbitrary subset of $\{1,2,\dots,\fn2\}$.
\end{proof}

\subsection{Major index and descent set}\label{sec:maj}

In this section we prove our main result. Its unrefined version states that the distribution of the major index over $321$-avoiding involutions is given by the central $q$-binomial coefficients. Recall that the $q$-binomial coefficients
are polynomials defined as $$\binom{n}{j}_q=\frac{(1-q^n)(1-q^{n-1})\dots(1-q^{n-j+1})}{(1-q^j)(1-q^{j-1})\dots(1-q)}.$$

\begin{theorem}\label{thm:maj}
For $n\ge1$,
\beq\label{eq:maj}\sum_{\pi\in\I_n(321)} q^{\maj(\pi)}=\binom{n}{\fn2}_q.\eeq
\end{theorem}

It is well known~\cite[Chapter~6]{Sta} that the coefficient of $q^m$ in the central $q$-binomial coefficient on the right hand side of~\eqref{eq:maj} equals the number of partitions of $m$ whose Young diagram fits inside a $\fn2 \times \cn2$ box.
In the rest of the paper, we denote this box by $B_n$, and we place coordinates on it so that its lower-left corner is at the origin and its upper-right corner is at $\urcorner$.
We write $\lambda\vdash m$ to denote that $\lambda$ is a partition of $m$, and we write $\lambda\subseteq B_n$ to denote that the Young diagram of $\lambda$ fits inside $B_n$.
By looking at the lower-right boundary of their Young diagrams, partitions satisfying the above two conditions can be interpreted as Grand Dyck paths from $(0,0)$ to $\urcorner$ with steps $N$ and $E$ such that the area of the region in $\mathbb{R}^2$
 inside $B_n$ that lies above the path is $m$, as shown in Figure~\ref{fig:hd}.

We will prove a refinement of Theorem~\ref{thm:maj}, which we state as Theorem~\ref{thm:main} below.
Given a partition $\lambda\vdash m$, we define its \emph{hook decomposition} $\hd(\lambda)=\{i_1,i_2,\dots,i_k\}$ (always written such that $i_1<\dots<i_k$) as follows.
The number of entries $k$ is the length of the side of the Durfee square of $\lambda$, that is, the largest value such that $\lambda_k\ge k$. The largest entry $i_k$ is the number of boxes in the largest hook of $\lambda$, which consists of the first column and first row of its Young diagram. Now remove the largest hook of $\lambda$ and define $i_{k-1}$ to be the number of boxes in the largest hook of the remaining Young diagram. Similarly, the remaining entries $i_j$ are defined recursively by peeling off hooks in the Young diagram. See Figure~\ref{fig:hd} for an example. Note that $i_{j}-i_{j-1}>1$ for all $j$ by construction.

\begin{figure}[htb]
  \begin{center}
    \begin{tikzpicture}[scale=0.5]
      \draw[fill=yellow!70] (0,-5) rectangle (1,0);
      \draw[fill=yellow!70] (1,-1) rectangle (4,0);
      \draw[fill=blue!25] (1,-5) rectangle (2,-1);
      \draw[fill=blue!25] (2,-2) rectangle (4,-1);
      \draw[fill=green!30] (2,-4) rectangle (3,-2);
      \draw[gray,thin] (0,-6) grid (6,0);
      \draw (0,-5) grid (1,0);
      \draw (1,-1) grid (4,0);
      \draw (1,-5) grid (2,-1);
      \draw (2,-2) grid (4,-1);
      \draw (2,-4) grid (3,-2);
      \draw[blue,ultra thick] (0,-6) coordinate(d0)
      -- ++(0,1) coordinate(d1)
      -- ++(1,0) coordinate(d2)
      -- ++(1,0) coordinate(d3)
      -- ++(0,1) coordinate(d4)
      -- ++(1,0) coordinate(d5)
      -- ++(0,1) coordinate(d6)
      -- ++(0,1)  coordinate(d7)
      -- ++(1,0) coordinate(d8)
      -- ++(0,1) coordinate(d9)
      -- ++(0,1)  coordinate(d10)
      -- ++(1,0)  coordinate(d11)
      -- ++(1,0) coordinate(d12);
      \foreach \x in {0,...,12} {
        \draw[blue,fill] (d\x) circle (2pt);
      } 
   \end{tikzpicture}
  \end{center}
  \caption{The Young diagram of the partition $\lambda=(4,4,3,3,2)\vdash 16$ inside the box $B_{12}$, and its hook decomposition $\hd(\lambda)=\{2,6,8\}$. The lower-right boundary of the Young diagram determines a Grand Dyck path from $(0,0)$ to $(6,6)$, which is highlighted in blue.}\label{fig:hd}
\end{figure}
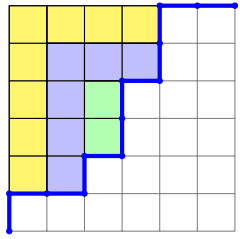

\begin{theorem}\label{thm:main}
Let $1\le i_1<i_2<\dots<i_k<n$, and let $m=i_1+\dots+i_k$.
There is a bijection
$$\{\pi\in\I_n(321):\Des(\pi)=\{i_1,i_2,...,i_k\}\}\longrightarrow
\{\lambda\vdash m:\hd(\lambda)=\{i_1,i_2,\dots,i_k\},\lambda\subseteq B_n\}.$$
\end{theorem}

Note that if $i_j-i_{j-1}=1$ for some $j$, then both sides in Theorem~\ref{thm:main} are empty sets, since two consecutive descents in a permutation would produce an occurrence of $321$.
For fixed $m$, taking the union over all subsets $\{i_1,\dots,i_k\}\subseteq[n-1]$ with $i_1+\dots+i_k=m$, Theorem~\ref{thm:main} gives a bijection
$$\{\pi\in\I_n(321):\maj(\pi)=m\}\longrightarrow
\{\lambda\vdash m:\lambda\subseteq B_n\},$$
so it implies Theorem~\ref{thm:maj}.

To prove Theorem~\ref{thm:main} we will use a sequence of bijections, as summarized in Figure~\ref{fig:bijections}.
The composition $\xi\circ\AR$ used in the proof of Theorem~\ref{thm:des} is not enough here, because it does not translate
the major index of the $321$-avoiding involution into the area above the Grand Dyck path. We will need an additional bijection mapping the statistic $\hd$ to $\Peak$, which we define next.

\begin{figure}
$$\begin{array}{ccccccc}
\I_n(321)& \stackrel{\AR}{\longrightarrow} & \P_n & \stackrel{\xi}{\longrightarrow} & \G_n & \stackrel{\hookpeak^{-1}}{\longrightarrow} & \{\lambda\subseteq B_n\}\\
\Des & \stackrel{\mathrm{Lem.~\ref{lem:DesPeak}}}{\leftrightarrow} & \Peak & \stackrel{\mathrm{Lem.~\ref{lem:peaksPG}}}{\leftrightarrow} & \Peak & \stackrel{\mathrm{Lem.~\ref{lem:hookpeak}}}{\leftrightarrow} & \hd
\end{array}$$
\caption{The statistic-preserving bijections used in the proof of Theorem~\ref{thm:main}.}\label{fig:bijections}
\end{figure}

\begin{lemma}\label{lem:hookpeak}
There is a bijection $\hookpeak$ from the set of partitions $\lambda$ inside $B_n$ to $\G_n$ such that, for all~$\lambda$,
$$\hd(\lambda)=\Peak(\hookpeak(\lambda)).$$
\end{lemma}

\begin{proof}
Given a partition $\lambda\subseteq B_n$, suppose that its Durfee square has side $k$ and that $\hd(\lambda)=\{i_1,\dots,i_k\}$.
Let $G\in\G_n$ be the path given by the boundary of the Young diagram of $\lambda$. Splitting $G$ at the point $M=(k,\fn2-k)$ we can write it as a concatenation $G=AB$, where $A$ and $B$ have $\fn2$ and $\cn2$ steps, respectively.
Counting the steps of $A$ starting at the point $M$, suppose that the $E$ steps occur at positions $1\le a_1<a_2<\dots<a_k\le \fn2$ (an $E$ step incident with $M$ would be considered to be at position $1$).
Similarly, counting the steps of $B$ starting at $M$, suppose that the $N$ steps occur at positions $1\le b_1<b_2<\dots<b_k\le \cn2$.
For each $1\le j\le k$, the hook of $\lambda$ of length $i_j$ is delimited by the $E$ step of $A$ at position $a_j$ and the $N$ step of $B$ in position $b_j$, from where it follows that $a_j+b_j=i_j+1$.

Define $\hookpeak(\lambda)$ to be the unique path in $\G_n$ that has peaks at coordinates $(b_j-1,a_j)$ for $1\le j\le k$.
The elements of its peak set are then $a_j+b_j-1=i_j$ for $1\le j\le k$, so $\Peak(\hookpeak(\lambda))=\hd(\lambda)$ as claimed.
Figure~\ref{fig:hookpeak} shows an example of this construction.

The map $\hookpeak$ is clearly invertible, because given a path in $\G_n$, the coordinates of its peaks determine the positions of the $N$ and $E$ steps in the boundary of the Young diagram of the corresponding partition inside $B_n$.
\end{proof}

\begin{figure}[htb]
  \begin{center}
    \begin{tikzpicture}[scale=0.6]
      \draw[gray,thin] (0,0) grid (7,7);
      \draw[blue,very thick] (0,0) coordinate(d0)
      -- ++(0,1) coordinate(d1)
      -- ++(1,0) coordinate(d2)
      -- ++(0,1) coordinate(d3)
      -- ++(0,1) coordinate(d4)
      -- ++(1,0) coordinate(d5)
      -- ++(1,0) coordinate(d6)
      -- ++(0,1)  coordinate(d7)
      -- ++(0,1) coordinate(d8)
      -- ++(1,0) coordinate(d9)
      -- ++(0,1)  coordinate(d10)
      -- ++(1,0)  coordinate(d11)
      -- ++(1,0) coordinate(d12)
      -- ++(1,0) coordinate(d13)
      -- ++(0,1) coordinate(d14);
      \draw (d7) circle (2pt);
      \draw (d7) node[above left] {$M$};
      \draw[yellow,very thick] (d1)--(d2);
      \draw[yellow,very thick] (d4)--(d6);
      \draw[yellow,very thick] (d7)--(d8);
      \draw[yellow,very thick] (d9)--(d10);
      \draw[yellow,very thick] (d13)--(d14);
      \draw[red,dotted,very thick] (d1)--(d2);
      \draw[red,dotted,very thick] (d4)--(d6);
      \draw[red,dotted,very thick] (d7)--(d8);
      \draw[red,dotted,very thick] (d9)--(d10);
      \draw[red,dotted,very thick] (d13)--(d14);
      \draw[red] (d2) node[below left] {\small $6$};
      \draw[red] (d5) node[below left] {\small $3$};
      \draw[red] (d6) node[below left] {\small $2$};
      \draw[red] (d7) node[above right] {\small $1$};
      \draw[red] (d9) node[above right] {\small $3$};
      \draw[red] (d13) node[above right] {\small $7$};
      \draw (9,3) node [label=$\hookpeak$] {$\mapsto$};
      \draw[gray,thin] (12,0) grid (19,7);
      \draw[blue,very thick] (12,0) coordinate(d0)
      -- ++(0,1) coordinate(d1)
      -- ++(0,1) coordinate(d2)
      -- ++(1,0) coordinate(d3)
      -- ++(1,0) coordinate(d4)
      -- ++(0,1) coordinate(d5)
      -- ++(1,0) coordinate(d6)
      -- ++(1,0)  coordinate(d7)
      -- ++(1,0) coordinate(d8)
      -- ++(1,0) coordinate(d9)
      -- ++(0,1)  coordinate(d10)
      -- ++(0,1)  coordinate(d11)
      -- ++(0,1) coordinate(d12)
      -- ++(1,0) coordinate(d13)
      -- ++(0,1) coordinate(d14);
      \draw (d2) node[above left] {$2$};
      \draw (d5) node[above left] {$5$};
      \draw (d12) node[above left] {$12$};
      \draw[red] (12,2) node[left] {\small $2\to$};
      \draw[red] (12,3) node[left] {\small $3\to$};
      \draw[red] (12,6) node[left] {\small $6\to$};
      \draw[red] (12,0) node[below] {\small $\stackrel{\uparrow}{0}$};
      \draw[red] (14,0) node[below] {\small $\stackrel{\uparrow}{2}$};
      \draw[red] (18,0) node[below] {\small $\stackrel{\uparrow}{6}$};
    \end{tikzpicture}
  \end{center}
  \caption{The bijection $\hookpeak$. In the path on the left, the numbers are the positions $a_j$ and $b_j$, which are then used in the path on the right to determine the locations of the peaks.}\label{fig:hookpeak}
\end{figure}
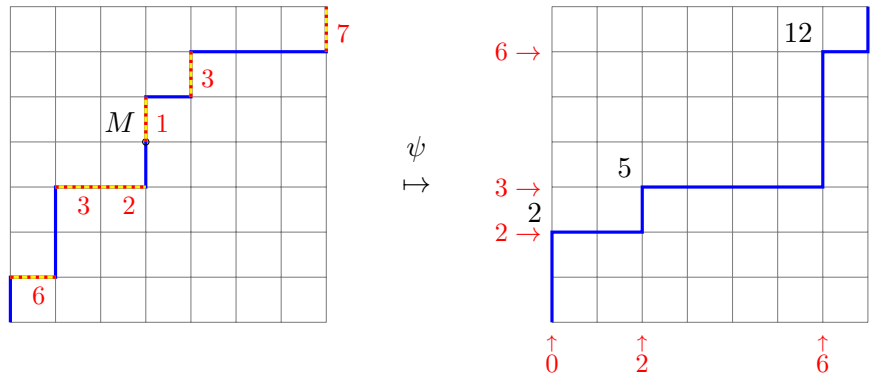

\begin{proof}[Proof of Theorem~\ref{thm:main}]
By Lemmas \ref{lem:DesPeak}, \ref{lem:peaksPG} and \ref{lem:hookpeak}, the composition of bijections $\hookpeak^{-1}\circ\xi\circ\AR:\I_n(123)\to\{\lambda\subseteq B_n\}$ maps the statistic $\Des$ to the statistic $\hd$. See Figure~\ref{fig:bijections} for a diagram of the preserved statistics, and Figure~\ref{fig:exbij} for an example.
\end{proof}

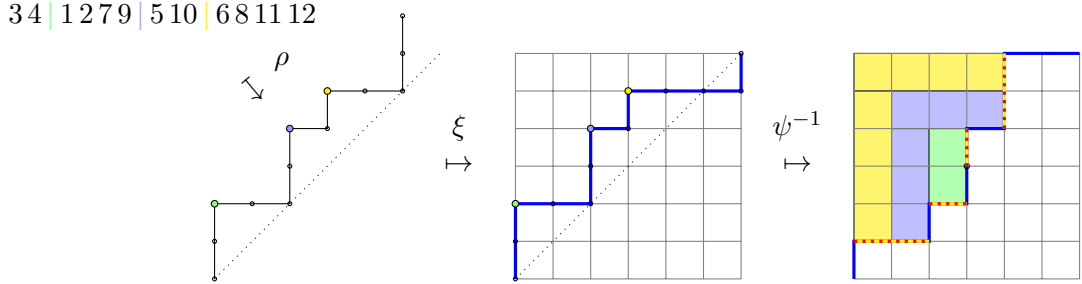
\begin{figure}[htb]
  \begin{center}
    \begin{tikzpicture}[scale=0.5]

      \draw (3,7) node[left] {$3\,4\,\textcolor{green!50}{|}\, 1\,2\,7\,9\,\textcolor{blue!40}{|}\,5\,10\,\textcolor{yellow!100}{|}\,6\,8\,11\,12$};
      \draw (1,5) node [label=$\AR$,rotate=-50] {$\mapsto$};
     \coordinate (origin) at (0,0);
      \draw (origin) coordinate(d0)
      -- ++(0,1) coordinate(d1)
      -- ++(0,1) coordinate(d2)
      -- ++(1,0) coordinate(d3)
      -- ++(1,0) coordinate(d4)
      -- ++(0,1) coordinate(d5)
      -- ++(0,1) coordinate(d6)
      -- ++(1,0)  coordinate(d7)
      -- ++(0,1) coordinate(d8)
      -- ++(1,0) coordinate(d9)
      -- ++(1,0)  coordinate(d10)
      -- ++(0,1)  coordinate(d11)
      -- ++(0,1) coordinate(d12);
      \foreach \x in {0,...,12} {
        \draw (d\x) circle (1.5pt);
      }
      \draw[fill=green!50] (d2) circle (2.5pt);
      \draw[fill=blue!40] (d6) circle (2.5pt);
      \draw[fill=yellow!100] (d8) circle (2.5pt);
      \draw[dotted] (origin)-- ++(6,6);
      \draw (origin) ++(6.5,3) node [label=$\xi$] {$\mapsto$};

      \coordinate (origin) at (8,0);
      \draw[gray,thin] (origin) grid ++(6,6);
      \draw[blue,very thick] (origin) coordinate(d0)
      -- ++(0,1) coordinate(d1)
      -- ++(0,1) coordinate(d2)
      -- ++(1,0) coordinate(d3)
      -- ++(1,0) coordinate(d4)
      -- ++(0,1) coordinate(d5)
      -- ++(0,1) coordinate(d6)
      -- ++(1,0)  coordinate(d7)
      -- ++(0,1) coordinate(d8)
      -- ++(1,0) coordinate(d9)
      -- ++(1,0)  coordinate(d10)
      -- ++(1,0)  coordinate(d11)
      -- ++(0,1) coordinate(d12);
      \foreach \x in {0,...,12} {
        \draw (d\x) circle (1.5pt);
      }
      \draw[fill=green!50] (d2) circle (2.5pt);
      \draw[fill=blue!40] (d6) circle (2.5pt);
      \draw[fill=yellow!100] (d8) circle (2.5pt);
      \draw[dotted] (origin)-- ++(6,6);
      \draw (origin) ++(7.5,3) node [label=$\hookpeak^{-1}$] {$\mapsto$};

      \coordinate (origin) at (17,0);
      \draw[fill=yellow!70] (origin)++(0,5) rectangle ++(4,1);
      \draw[fill=yellow!70] (origin)++(0,1) rectangle ++(1,4);
      \draw[fill=blue!25] (origin)++(1,4) rectangle ++(3,1);
      \draw[fill=blue!25] (origin)++(1,1) rectangle ++(1,3);
      \draw[fill=green!30] (origin)++(2,2) rectangle ++(1,2);
      \draw[gray,thin] (origin) grid ++(6,6);
      \draw[blue,very thick] (origin) coordinate(d0)
      -- ++(0,1) coordinate(d1)
      -- ++(1,0) coordinate(d2)
      -- ++(1,0) coordinate(d3)
      -- ++(0,1) coordinate(d4)
      -- ++(1,0) coordinate(d5)
      -- ++(0,1) coordinate(d6)
      -- ++(0,1)  coordinate(d7)
      -- ++(1,0) coordinate(d8)
      -- ++(0,1) coordinate(d9)
      -- ++(0,1)  coordinate(d10)
      -- ++(1,0)  coordinate(d11)
      -- ++(1,0) coordinate(d12);
      \draw (d6) circle (2pt);
      \draw[yellow,very thick] (d1)--(d3);
      \draw[yellow,very thick] (d4)--(d5);
      \draw[yellow,very thick] (d6)--(d7);
      \draw[yellow,very thick] (d8)--(d10);
      \draw[red,dotted,very thick] (d1)--(d3);
      \draw[red,dotted,very thick] (d4)--(d5);
      \draw[red,dotted,very thick] (d6)--(d7);
      \draw[red,dotted,very thick] (d8)--(d10);
    \end{tikzpicture}
  \end{center}
  \caption{An example of the sequence of bijections in the proof of Theorem~\ref{thm:main}. Note that the resulting partition $(4,4,3,3,2)$ has hook decomposition $\{2,6,8\}$, which agrees with the descent set of the $321$-avoiding involution that we started from.}\label{fig:exbij}
\end{figure}

As a consequence of Theorem~\ref{thm:main}, we obtain the following refinement of Theorems~\ref{thm:des} and~\ref{thm:maj}.

\begin{corollary}
For every $0\le k<n$,
$$\sum_{\substack{\pi\in\I_n(321)\\ \des(\pi)=k}}q^{\maj(\pi)}=q^{k^2}\binom{\cn2}{k}_q\binom{\fn2}{k}_q.$$
\end{corollary}

\begin{proof}
By the bijection in Theorem~\ref{thm:main}, the left hand side is the generating polynomial for Young diagrams inside $B_n$ with Durfee square of side $k$ with respect to area. The formula on the right hand side follows by decomposing such diagrams as in Figure~\ref{fig:partitiondecomposition} and using that $\binom{a+b}{a}_q$ is the generating polynomial for Young diagrams inside an $a\times b$ box with respect to area.
\end{proof}

\section{Consequences}\label{sec:consequences}

\subsection{Ascent sets on $123$-avoiding involutions}\label{sec:123}

Our work on descents on $321$-avoiding involutions easily extends to describe the distribution of the ascent set and comajor index on $123$-avoiding involutions. Note that this does not follow from any trivial symmetries on permutations, since those that take $321$-avoiding permutations to $123$-avoiding ones do not preserve the property of being an involution.

Given $\pi\in\I_n$, the Robinson-Schensted algorithm associates to it a pair $(Q,Q)$ of indentical standard Young tableaux of size $n$. Let $Q^T$ be the standard Young tableaux obtained by transposing $Q$, and let
$\pi^{T}\in\I_n$ be the preimage of the  pair $(Q^{T},Q^{T})$ under the Robinson-Schensted correspondence. The map $\pi\mapsto \pi^{T}$ is a bijection from $\I_n$ to itself.
As an example, the image of the involution $8\,6\,12\,11\,5\,2\,10\,1\,9\,7\,4\,3$ is $3\,4\,1\,2\,7\,9\,5\,10\,6\,8\,11\,12$.

The following result, which first appeared in \cite{Strehl} (see \cite{BBS2} for a detailed exposition), is an immediate consequence of the fact that the descent set of a permutation equals the descent set of its recording tableau.

\begin{prop}[\cite{Strehl}]\label{prop:Strehl}
For every $\pi\in\I_n$, we have $\Asc(\pi)=\Des(\pi^{T})$.
\end{prop}

Consider now the set $\I_n(12\ldots k)$ of involutions in $\S_n$ that avoid the pattern $12\ldots k$.
By Schensted's Theorem~\cite{Sche}, the Robinson-Schensted algorithm associates to each $\pi\in \I_n(12\ldots k)$ a standard Young tableau $Q$ with at most $k-1$ columns. Equivalently, the tableau $Q^T$ has at most $k-1$ rows,
and so its corresponding involution $\pi^{T}$ avoids $k\ldots 21$. It follows that the map $\pi\mapsto \pi^{T}$ induces a bijection between $\I_n(12\ldots k)$ and $\I_n(k\dots 21)$ with the property that $\Asc(\pi)=\Des(\pi^T)$.
For $k=3$, the following results are now equivalent to Theorems~\ref{thm:main} and~\ref{thm:maj}, respectively.

\begin{corollary}\label{thm:newmain}
Let $1\le i_1<i_2<\dots<i_k<n$, and let $m=i_1+\dots+i_k$.
There is a bijection
$$\{\pi\in\I_n(123):\Asc(\pi)=\{i_1,i_2,...,i_k\}\}\longrightarrow
\{\lambda\vdash m:\hd(\lambda)=\{i_1,i_2,\dots,i_k\},\lambda\subseteq B_n\}.$$
\end{corollary}

\begin{corollary}\label{thm:newmaj}
For $n\ge1$,
$$\sum_{\pi\in\I_n(123)} q^{\comaj(\pi)}=\binom{n}{\fn2}_q.$$
\end{corollary}

\subsection{Large $n$}\label{sec:largen}

Even though the simplest proof of Theorem~\ref{thm:main} that we know uses a composition of three non-trivial bijections, as described in Section~\ref{sec:maj}, it is interesting to note that for $n$ large enough (relative to $i_k$), there is a simpler proof. It will be more convenient to work with $123$-avoiding involutions, so we will consider Corollary~\ref{thm:newmain}, which is equivalent to Theorem~\ref{thm:main} via the map in Proposition~\ref{prop:Strehl}.
Next we sketch a direct proof of Corollary~\ref{thm:newmain} for large $n$ which does not use lattice paths.

Fix $1\le i_1<i_2<\dots<i_k\ll n$. It is easy to see that any involution $\pi\in\I_n(123)$ is uniquely determined by the positions of its \LTR minima, that is, the set of indices $i$ such that $\pi(i)<\pi(j)$ for all $j<i$. Note also that since $\pi$ avoids $123$, we have $i\in\Asc(\pi)$ if and only if $i$ is a \LTR minimum and $i+1$ is not. It follows that if $\Asc(\pi)=\{i_1,i_2,...,i_k\}\neq\emptyset$, then the set of \LTR minima of $\pi$ has the form
\begin{equation}[1,i_1]\cup[a_2,i_2]\cup[a_3,i_3]\cup\dots\cup[a_k,i_k]\cup [b,n],\label{eq:ltr}\end{equation}
where $i_{j-1}+2\le a_j\le i_j$ for each $2\le j\le k$, and $n-i_1+2\le b\le n+1$. In fact, for $n$ large enough (more precisely, $n\geq 2(i_k-k+1)$), any such choice of the $a_j$ and $b$ yields a valid $\pi\in\I_n(123)$.
Thus,
$$|\{\pi\in\I_n(123):\Asc(\pi)=\{i_1,i_2,...,i_k\}\}|=i_1\prod_{j=2}^k (i_j-i_{j-1}-1).$$
On the other hand, if $n\geq 2(i_k-k+1)$, then every partition $\lambda$ with $\hd(\lambda)=\{i_1,i_2,\dots,i_k\}\neq\emptyset$ has a Young diagram that fits inside $B_n$, and so
$$|\{\lambda:\hd(\lambda)=\{i_1,i_2,\dots,i_k\},\lambda\subseteq B_n\}|=|\{\lambda:\hd(\lambda)=\{i_1,i_2,\dots,i_k\}\}|=i_1\prod_{j=2}^k (i_j-i_{j-1}-1),$$
since we can construct such a partition by first choosing among the $i_1$ ways to bend the innermost hook (the one of size $i_1$), then choosing among the $i_2-i_1-1$ ways to place the hook of size $i_2$ around the hook of size $i_1$, and so on, placing the hooks from the inside to the outside. The degenerate case $\Asc(\pi)=\emptyset$ corresponds to the empty partition.

The above argument, which proves Corollary~\ref{thm:newmain} when $n$ is large, breaks down for small $n$, and there does not seem to be a natural way to fix it.
It is not true for any $n$ that every choice of \LTR minima of the form~\eqref{eq:ltr}
is realized by an involution in $\I_n(123)$, nor that every placement of hooks as described above will produce a Young diagram that fits inside $B_n$.

\medskip

Next we discuss some consequences of Theorems~\ref{thm:maj} and Theorem~\ref{thm:main} when $n$ is large. Using that every partition of $m$ fits inside $B_n$ for $n$ large enough, Theorem~\ref{thm:maj} implies the following.

\begin{corollary}
For $n\ge 2m$,
$$|\{\pi\in\I_n(321):\maj(\pi)=m\}|=p(m),$$
where $p(m)$ is the number of partitions of $m$.
\end{corollary}

We remark that this result is somewhat reminiscent of Propositions 11 and 15 in~\cite{CJS}, which give formulas in terms of $m$ for counting pattern-avoiding permutations in $\S_n$ with $m$ inversions when $n$ is large enough.

Along the same lines, Theorem~\ref{thm:main}
can be used to obtain the generating function for $321$-avoiding involutions with $k$ descents according to their descent set. If $S=\{i_1,i_2,...,i_k\}$ with $i_1<\dots<i_k$, we write $\x^S=x_1^{i_1}\dots x_k^{i_k}$.

\begin{corollary}
\begin{align*}\lim_{n\to\infty}\sum_{\pi\in\I_n(321)} \x^{\Des(\pi)}&=\sum_{k\ge0}\frac{x_1 x_2^{3} \dots x_k^{2k-1}}{(1-x_k)^2 (1-x_{k-1} x_k)^2 \dots (1-x_1 x_2 \dots x_k)^2},\\
\lim_{n\to\infty}\sum_{\pi\in\I_n(321)} t^{\des(\pi)} q^{\maj(\pi)}&=\sum_{k\ge0} \frac{t^k q^{k^2}}{(1-q)^2 (1-q^2)^2 \dots (1-q^k)^2}.
\end{align*}
\end{corollary}

\begin{proof}
Let $1\le i_1<i_2<\dots<i_k$. For $n\geq 2(i_k-k+1)$, Theorem~\ref{thm:main} gives a bijection between the set of involutions $\pi\in\I_n(321)$ with $\Des(\pi)=\{i_1,i_2,...,i_k\}$ and the set of partitions $\lambda$ with
$\hd(\lambda)=\{i_1,i_2,\dots,i_k\}$, since the Young diagram of every such partition fits inside $B_n$.
The generating function for partitions $\lambda$ with Durfee square of side $k$ according to their hook decomposition is
\beq\label{eq:durfee}\sum_{\lambda}\x^{\hd(\lambda)}=\frac{x_1 x_2^{3} \dots x_k^{2k-1}}{(1-x_k)^2 (1-x_{k-1} x_k)^2 \dots (1-x_1 x_2 \dots x_k)^2},\eeq
since such partitions can be decomposed into pairs of partitions with at most $k$ parts, each contributing $\left((1-x_k)(1-x_{k-1} x_k)\dots (1-x_1 x_2 \dots x_k)\right)^{-1}$, attached to the $k\times k$ Durfee square, which contributes $x_1 x_2^{3} \dots x_k^{2k-1}$. This decomposition is illustrated in Figure~\ref{fig:partitiondecomposition}.

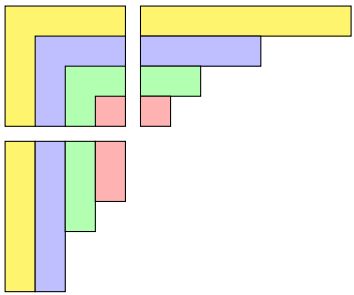
\begin{figure}[htb]
  \begin{center}
    \begin{tikzpicture}[scale=0.4]
      \draw[fill=yellow!70] (0,0) rectangle (4,-4);
      \draw[fill=blue!25] (1,-1) rectangle (4,-4);
      \draw[fill=green!30] (2,-2) rectangle (4,-4);
      \draw[fill=red!30] (3,-3) rectangle (4,-4);
      \draw[fill=yellow!70] (4.5,0) rectangle (11.5,-1);
      \draw[fill=blue!25] (4.5,-1) rectangle (8.5,-2);
      \draw[fill=green!30] (4.5,-2) rectangle (6.5,-3);
      \draw[fill=red!30] (4.5,-3) rectangle (5.5,-4);
      \draw[fill=yellow!70] (0,-4.5) rectangle (1,-9.5);
      \draw[fill=blue!25] (1,-4.5) rectangle (2,-9.5);
      \draw[fill=green!30] (2,-4.5) rectangle (3,-7.5);
      \draw[fill=red!30] (3,-4.5) rectangle (4,-6.5);
    \end{tikzpicture}
  \end{center}
  \caption{Decomposition of the Young diagram of a partition.}\label{fig:partitiondecomposition}
\end{figure}

The generating function for involutions with $k$ descents with respect to $\maj$, or equivalently partitions with Durfee square of side $k$ with respect to their size, is obtained by 
substituting $x_j=q$ for all $j$ in Equation~\eqref{eq:durfee}. Multiplying by $t^k$ and summing over all $k\ge0$ we obtain the second formula in the statement.

\end{proof}

\subsection{Double avoidance}

In closing this section, we look at the distribution of descents and major index on involutions that avoid $321$ and an additional pattern $\tau\in\S_3$. The set of such involutions of length $n$ is denoted by $\I_n(321,\tau)$.
We only consider the cases $\tau=312$ and $\tau=213$, since the set $\I_n(321,123)$ is empty for $n\geq 6$, and $\I_n(321,231)$ and $\I_n(321,132)$ correspond via the usual reverse-complement map to the studied cases.

For the case $\tau=213$, it is easy to verify that $\pi\in\I_n(321,213)$ if and only if $\pi=p(p+1)\ldots n12\ldots(p-1)$ for some $1\leq p\leq n$. Hence, $\pi$ has either one descent in position $n+1-p$ (if $p\neq1$) or no descents at all. It follows that
$$\sum_{\pi\in\I_n(321,213)} q^{\maj(\pi)}=\frac{1-q^n}{1-q}.$$

Now we consider the case $\tau=312$. It is easy to see that if $\pi\in\I_n(321,312)$, then for every $1\leq i\leq n$, the
entry $\pi(i)$ is either the smallest or the second smallest among $\pi(i),\pi(i+1),\ldots,\pi(n)$.
This condition, together with the fact that $\pi$ is an involution, implies that $\pi$ is a direct sum $\pi=\sigma_1\oplus\sigma_2\ldots\oplus\sigma_t$, where each $\sigma_j$ equals $1$ or $21$
(recall that this means that $\pi$ is a juxtaposition of words order-isomorphic to either $1$ or $21$, where the entries in each word are smaller than the entries in the next word).
These are called {\em Fibonacci permutations} in~\cite{CJS}.
It follows that the generating function with respect to the number of descents is
$$\sum_{n\ge0}\sum_{\pi\in\I_n(321,312)} t^{\des(\pi)} x^n=\frac{1}{1-x-tx^2},$$
which gives the triangle of coefficients of Fibonacci polynomials (see \cite[seq. A011973]{oeis}).

The above observation also yields a recurrence for the polynomials $p_n(q):=\sum_{\pi\in\I_n(321,312)} q^{\maj(\pi)}$,
since every $\pi\in\I_n(321,312)$ can be obtained by appending $n$ or $n(n-1)$ to an involution of length $n-1$ or $n-2$, respectively. We get that
$$p_n(q)=p_{n-1}(q)+q^{n-1}p_{n-2}(q)$$
for $n\ge2$, with initial conditions $p_0(q)=p_1(q)=1$ (see \cite[seq. A127836]{oeis}).

\section{Descent sets on $321$-avoiding permutations}\label{sec:321permutations}

Some of the ideas used above to study descent sets on $321$-avoiding involutions can be applied to $321$-avoiding permutations, even though we do not obtain nice formulas analogous to Theorems~\ref{thm:maj} and~\ref{thm:main}.

The distribution of the major index on $321$-avoiding permutations has been studied by Cheng {\it et al}., who in~\cite[Thm 6.2]{CEKS} give a recurrence for the generating polynomial for the statistic $\maj$ on $\S_n(321)$.
Here we are interested in the distribution of the whole descent set. The following result gives a simple description. In the rest of this section we denote the $m$-th Catalan number by $C_m=\frac{1}{m+1}\binom{2m}{m}$.

\begin{theorem}\label{thm:s321}
Let $S\subseteq[n-1]$. Then
$$|\{\pi\in\S_n(321):\Des(\pi)\supseteq S\}|=\begin{cases} C_{n-|S|} & \mbox{if $S$ contains no two consecutive elements,} \\ 0 &\mbox{otherwise.}\end{cases}$$
\end{theorem}

\begin{proof}
We apply the following bijection between $\S_n(321)$ and $\D_n$ from \cite[Section~3]{EliPak}, which is an extension of the bijection $\rho$ used above for involutions.
First, the Robinson-Schensted correspondence gives a bijection between $\S_n(321)$ and pairs $(P,Q)$ of standard Young tableaux of the same shape having $n$ boxes and at most two rows.
We can interpret $P$ and $Q$ as Dyck path prefixes ending that the same height, where the entries on the first row determine the positions of the $N$ steps, and the entries on the second row determine the positions of the $E$ steps. These two prefixes can be combined into a Dyck path by taking the prefix corresponding to the recording tableau $Q$ followed by the reversal of the prefix corresponding to the insertion tableau $P$.

Recall from the proof of Lemma~\ref{lem:DesPeak} that the Robinson-Schensted algorithm maps the descent set of the permutation to the descent set of the recording tableau $Q$, which in turn becomes the peak set of the Dyck path prefix associated to $Q$. Thus, if our bijection maps $\pi\in\S_n(321)$ to $D\in\D_n$, then
$\Des(\pi)$ is the set of peak positions in the first half of $D$, that is, $\Des(\pi)=\Peak(D)\cap[n-1]$.

It follows that, for any $S\subseteq[n-1]$, the number of permutations in $\S_n(321)$ with descent set containing $S$ equals the number of paths in $\D_n$ with peak set containing $S$. If $S$ contains two consecutive elements, this set is clearly empty. Otherwise, there is a simple bijection between $\{D\in D_n:\Peak(D)\supseteq S\}$ and $\D_{n-|S|}$: given a path in the first set, remove the peaks $NE$ in positions given by $S$. This construction gives a Dyck path with $2n-2|S|$ steps, and it is clearly invertible, since $S$ keeps track of the positions from where peaks were removed.
\end{proof}

Next we use Theorem~\ref{thm:s321} to obtain a summation formula and a recurrence for the generating polynomial for $321$-avoiding permutations with respect to the descent set.
If $S$ is a set of positive integers, we use the notation $\x_S=\prod_{j\in S}x_j$.

\begin{corollary}\label{cor:s321} We have that
\beq\label{eq:cors321}\sum_{\pi\in\S_n(321)}\x_{\Des(\pi)}=\sum_T \left(C_{n-|T|}\prod_{j\in T}(x_j-1)\right),\eeq
where $T$ ranges over all subsets of $[n-1]$ with no two consecutive elements.
\end{corollary}

\begin{proof}
For any set $S$, is is clear that $\prod_{j\in S}(1+y_j)=\sum_{T\subseteq S}\prod_{j\in T}y_j$. Making the substitution $y_j=x_j-1$ yields $$\x_S=\sum_{T\subseteq S}\prod_{j\in T}(x_j-1).$$
Using this identity, we get
$$\sum_{\pi\in\S_n(321)}\x_{\Des(\pi)}=\sum_{\pi\in\S_n(321)}\sum_{T\subseteq \Des(\pi)}\prod_{j\in T}(x_j-1)=\sum_{T\subseteq[n-1]}\sum_{\substack{\pi\in\S_n(321)\\ \Des(\pi)\supset T}}\prod_{j\in T}(x_j-1),$$
which equals the right hand of~\eqref{eq:cors321} by Theorem~\ref{thm:s321}.
\end{proof}

Extracting the coefficient of $\x_S$ in Corollary~\ref{cor:s321}, it follows that for any $S\subseteq[n-1]$ with no consecutive elements,
\beq\label{eq:mobius}|\{\pi\in\S_n(321):\Des(\pi)=S\}|=\sum_T (-1)^{|T|-|S|} C_{n-|T|},\eeq
where now $T$ ranges over all subsets of $[n-1]$ containing $S$ and having no two consecutive elements. Equation~\eqref{eq:mobius} can also be obtained directly from Theorem~\ref{thm:s321} using M\"obius inversion.

For $n,m\ge0$, let $$A_{n,m}(\x)=\sum_T \left(C_{m-|T|}\prod_{j\in T}(x_j-1)\right),$$
where $T$ ranges over all subsets of $[n-1]$ with no two consecutive elements. Note that $A_{n,n}(\x)$ is the right hand side of Equation~\eqref{eq:cors321}. Separating terms depending on whether $n-1\in T$ or not, we obtain the following recurrence for $A_{n,m}(\x)$.

\begin{corollary}\label{cor:321rec}
For $n\ge2$ and $m\ge1$,
$$A_{n,m}(\x)=(x_{n-1}-1)A_{n-2,m-1}(\x)+A_{n-1,m}(\x),$$
with initial conditions $A_{0,m}(\x)=A_{1,m}(\x)=C_m$ for $m\ge0$ and $A_{n,0}(\x)=1$ for $n\ge2$.
\end{corollary}

Note that setting $x_j=q^j$ for all $j$ in Corollary~\ref{cor:s321} we have $A_{n,n}(q,q^2,q^3,\ldots)=\sum_{\pi\in\S_n(321)}q^{\maj{\pi}}$, and so the recurrence in Corollary~\ref{cor:321rec} can be used to compute
this polynomial. A different and arguably more complicated recurrence is given in \cite[Thm. 6.2]{CEKS}. It would be interesting to find a simple formula enumerating permutations in $\S_n(321)$ with a given major index in the spirit of Theorem~\ref{thm:maj}. A helpful tool might be the bijection between $\S_n(321)$ and $\D_n$ described in the proof of Theorem~\ref{thm:s321}, which maps the statistic $\maj$ on $\S_n(321)$ to the sum of the peak positions of peaks in the first half of the corresponding Dyck path.

\section{Appendix: A non-bijective proof of Theorem~\ref{thm:maj}}\label{sec:Stanley}

In this appendix we discuss a non-bijective proof of Theorem~\ref{thm:maj} that was communicated to us by Richard Stanley.

As discussed in Sections~\ref{sec:des} and~\ref{sec:123}, the Robinson-Schensted correspondence gives a descent-set-preserving bijection between $\I_n(k\dots21)$ and the set $\SYT_n^{k-1}$ of standard Young tableaux with $n$ boxes at most $k-1$ rows. Recall that the major index of a standard Young tableaux is defined as the sum of its descents. It follows from \cite[Prop. 7.19.11]{EC2} that
\beq\label{eq:shur}\sum_{\pi\in\I_n(k\dots21)} q^{\maj(\pi)}=\sum_{T\in\SYT_n^{k-1}} q^{\maj(T)}=(1-q)(1-q^2)\cdots(1-q^n)\,\sum_\lambda s_\lambda(1,q,q^2,\ldots),\eeq
where $\lambda$ ranges over all partitions of $n$ with at most $k-1$ parts, and $s_\lambda$ denotes a Schur function.
In the case $k=3$, if follows from \cite[Ex. 7.16a]{EC2} (see also~\cite{BN}) that
$$\sum_\lambda s_\lambda=h_{\fn2} h_{\cn2},$$
where now the sum if over partitions of $n$ with at most $2$ parts, and $h_m=\sum_{1\le i_1\le i_2\le \dots\le i_m} x_{i_1}x_{i_2}\dots x_{i_m}$ denotes a complete homogeneous symmetric function.
Since $$h_m(1,q,q^2,\dots)=\frac{1}{(1-q)(1-q^2)\dots(1-q^m)},$$ Equation~\eqref{eq:shur} gives
$$\sum_{\pi\in\I_n(321)} q^{\maj(\pi)}=\frac{(1-q)(1-q^2)\dots(1-q^n)}{(1-q)(1-q^2)\dots(1-q^{\fn2})(1-q)(1-q^2)\dots(1-q^{\cn2})}=\binom{n}{\fn2}_q,$$
recovering Theorem~\ref{thm:maj}.

Although this non-bijective method cannot be used to prove the more general Theorem~\ref{thm:main}, it can in principle be extended to enumerate $k\dots21$-avoiding involutions for larger values of $k$. 

\subsection*{Acknowledgements}

The first two authors were partially supported by University of Bologna, funds for selected research topics, and by P.R.I.N. of M.I.U.R., Italy. The third author was partially supported by grant DMS-1001046 from the NSF,
by grant \#280575 from the Simons Foundation,
and by grant H98230-14-1-0125 from the NSA. We are grateful to Richard Stanley for providing the non-bijective proof in Section~\ref{sec:Stanley}. Finally, we point out that, after submission of this paper, a different proof of Theorem~\ref{thm:maj} has been found by Dahlberg and Sagan~\cite{DSagan}.

\end{document}